\documentclass[11pt]{article}
\usepackage{verbatim,url,enumerate,color,paralist}
\usepackage{amsmath,amsfonts,amsthm}
\usepackage{epsfig,amssymb,amstext,xspace}
\usepackage{algorithm}
\usepackage{algorithmicx,algpseudocode}
\usepackage{fullpage}

\newtheorem{theorem}{Theorem}[section]
\newtheorem{lemma}[theorem]{Lemma}
\newtheorem{corollary}[theorem]{Corollary}
\newtheorem{definition}[theorem]{Definition}

\newtheorem{claim}[theorem]{Claim}
\newtheorem{fact}[theorem]{Fact}

\newtheorem*{lemma-order}{Lemma~\ref{lem:order}}
\newtheorem*{claim-subinstance}{Claim~\ref{claim: small cover-sets in subinstance}}
\newtheorem*{lemma-solution}{Lemma~\ref{lem: sa feasible solution}}

\newcommand{\LS}{{\rm LS}}

\newcommand{\classP}{{\mathrm{P}}}
\newcommand{\classNP}{{\mathrm{NP}}}

\newcommand{\Xomit}[1]{}

\newcommand{\opt}{{\sf OPT}}
\newcommand{\Real}{{\mathbb R}}

\newcommand{\powerset}{\mathcal{P}}
\newcommand{\iLS}{\textsf{LS}}  
\newcommand{\iLSp}{\textsf{LS}$_+$}

\newcommand{\sa}{\textsf{Sherali-Adams}}
\newcommand{\ls}{\textsf{Lov\'asz-Schrijver}}
\newcommand{\la}{\textsf{Lasserre}}
\newcommand{\setc}{\textsc{Set Cover}}
\def\knapsack {\textsc{Knapsack}}
\def\threesat{\textsc{Max-3SAT}}

\def\chrom{\textsc{Chromatic Number}}
\def\clique{\textsc{Clique}}
\def\ug {\textsc{Unique Games}}
\def\gst {\textsc{Group Steiner Tree}}
\newcommand{\iLP}{\textsf{LP}}

\newcommand{\iSDP}{\textsf{SDP}}
\newcommand{\eps}{\varepsilon}

\def\be {{\bf e}}

\def\bx {{\bf x}}
\def\by {{\bf y}}
\def\bz {{\bf z}}

\def\ba {{\bf a}}
\newcommand{\OPT}{\textsc{OPT}}

\begin{document}
\title{Understanding Set Cover: Sub-exponential Time Approximations and Lift-and-Project Methods}

\author{Eden Chlamt\'a\v{c}\thanks{Research supported in part by an ERC Advanced grant. 
Email: \texttt{chlamtac@cs.bgu.ac.il}} \\ Ben Gurion University
\and Zachary Friggstad\thanks{Email: \texttt{\{zfriggstad,k2georgiou\}@math.uwaterloo.ca}} \\ University of Waterloo
\and Konstantinos Georgiou$^\dagger$ \\ University of Waterloo}

\begin{titlepage}
\maketitle
\thispagestyle{empty}


\begin{abstract}
Recently, Cygan, Kowalik, and Wykurz~[IPL 2009] gave sub-exponential-time approximation algorithms for the \setc\ problem with approximation ratios better than $\ln n$. In light of this result, it is natural to ask whether such improvements can be achieved using lift-and-project methods. We present a simpler combinatorial algorithm which has nearly the same time-approximation tradeoff as the algorithm of Cygan et al., and which lends itself naturally to a lift-and-project based approach.


At a high level, our approach is similar to the recent work of 
 Karlin, Mathieu, and Nguyen [IPCO 2011], who examined a known PTAS for \knapsack\ (similar to our combinatorial \setc\ algorithm) and its connection to hierarchies of 
\iLP\ and 
\iSDP\ relaxations for \knapsack. For \setc, 
 we show that, indeed, using the trick of ``lifting the objective function", we can match the performance of our combinatorial algorithm using the {\iLP} hierarchy of Lov\'asz and Schrijver.
We also show that this trick is essential: even in the stronger {\iLP} hierarchy of Sherali and Adams, the integrality gap remains at least $(1-\eps)\ln n$ at level $\Omega(n)$ (when the objective function is not lifted).

As shown by Aleknovich, Arora, and Tourlakis~[STOC 2005], \setc\ relaxations stemming from {\iSDP} hierarchies (specifically, $\LS_+$) have similarly large integrality gaps. 
 This stands in contrast to \knapsack, where Karlin et al.\ showed that the (much stronger) \la\ {\iSDP} hierarchy reduces the integrality gap to $(1+\eps)$ at level $O(1)$. For completeness, we show that $\LS_+$ also reduces the integrality gap for \knapsack\ to $(1+\eps)$. 
 This result may be of independent interest, as our $\LS_+$-based rounding and analysis are rather different from those of Karlin et al., and to the best of our knowledge this is the first explicit demonstration of such a reduction in the integrality gap of $\LS_+$ relaxations after few rounds.

\end{abstract}
\end{titlepage}
\section{Introduction}

The \setc\ problem is one of the most fundamental and well-studied problems in approximation algorithms, and was one of Karp's original 21 NP-complete problems~\cite{karp}. It can be stated quite plainly: given a finite set $X$ of $n$ items and a collection $\mathcal S \subseteq 2^X$ of $m$ subsets
of $X$ called ``cover-sets'', the \setc\ problem on instance $(X, \mathcal S)$
is the problem of finding the smallest collection $\mathcal C$ of cover-sets in $\mathcal S$ such that $X = \bigcup_{S \in \mathcal C} S$.
That is, every item in $X$ must appear in at least one cover-set in $\mathcal C$. Here, we will consider the minimum cost (or weighted) version of the problem, where each
cover-set $S$ has a nonnegative cost $c(S)$, and the goal is to find a collection of cover-sets with minimum total cost, subject to the above constraint.

As is well known, the problem can be approximated within a logarithmic factor. For instance, Johnson~\cite{johnson} showed that for uniform costs,
the simple greedy algorithm that iteratively chooses the cover-set containing the maximum number of uncovered elements
gives an $H_n$-approximation (where $H_n=\ln n + O(1)$ is the $n$'th harmonic number $\sum_{k=1}^{\lfloor n\rfloor} 1/k$). %
Later, Lov\'asz~\cite{lovasz} showed that the cost of the solution found by this greedy algorithm is at most an $H_n$-factor larger than the optimum value of the natural linear programming (LP) relaxation.

Chv\'atal~\cite{chvatal} extended these results to a greedy algorithm for the weighted case. He also proved that the approximation guarantee of this algorithm is actually only $H_b$ where
$b$ is the size of the largest cover-set in $\mathcal S$, and moreover that this algorithm gives an $H_b$-factor approximation relative to the optimum of the natural {\iLP} relaxation (thus %
extending the integrality gap bound of Lov\'asz~\cite{lovasz}).
Slav\'ik~\cite{slavik} refined the lower-order terms in the analysis of the greedy algorithm's approximation guarantee, giving a tight bound of $\ln n - \ln\ln n + O(1)$. 
Srinivasan~\cite{srinivasan} also improved the lower-order terms in the upper bound on the integrality gap. On the other hand, 
 the integrality gap of the standard {\iLP} relaxation for \setc\ is at least $(1-o(1))\ln n$. Recently, Cygan, Kowalik, and Wykurz \cite{cygan:kowalik:wykurz} demonstrated that Set Cover can be approximated
within $(1 - \epsilon)\cdot \ln n + O(1)$ in time $2^{n^\epsilon + O(\log m)}$. It is interesting to note that this time-approximation tradeoff is essentially optimal assuming Moshkovitz's Projection Games Conjecture~\cite{moshkovitz} and the Exponential Time Hypothesis (ETH)~\cite{IP01}.

From the hardness perspective, 
 Feige~\cite{feige} showed that for every constant $\eps > 0$, there is no
$(1 - \eps)\ln n$-approximation algorithm for \setc\ unless all problems in NP can be solved deterministically in time $n^{O(\log\log n)}$. 
 To date, the strongest hardness of approximation for \setc\ assuming only $\classP\neq\classNP$ gives a $c\ln n$-hardness  %
 for $c \approx 0.2267$ (Alon et al.~\cite{alon:moshkovitz:safra}).

\subsection{Hierarchies of convex relaxations and connection to \knapsack}

One of the most powerful and ubiquitous tools in approximation algorithms has been the use of mathematical programming relaxations, such as linear programming (\iLP) and semidefinite programming ({\iSDP}). The common approach is as follows: solve a convex (\iLP\ or {\iSDP}) relaxation for the 0-1 program, and ``round" the relaxed solution to give a (possibly suboptimal) feasible 0-1 solution. Since the approximation ratio is usually analyzed by comparing the value of the relaxed solution to the value of the output (note that the 0-1 optimum is always sandwiched between these two), a natural obstacle is the worst case ratio between the relaxed optimum and the 0-1 optimum, known as the {\em integrality gap}.

While for many problems, this approach gives optimal approximations (e.g., Raghavendra~\cite{ragh08} shows this for all CSPs, assuming the \ug\ Conjecture), there are still many cases where natural {\iLP} and {\iSDP} relaxations have large integrality gaps. This limitation can be circumvented by considering more powerful relaxations. In particular, Sherali and Adams~\cite{SA90} Lov\'asz and Schrijver~\cite{LS91}, and Lasserre~\cite{Las} each have devised different systems, collectively known as {\em hierarchies} or {\em lift-and-project techniques}, by which a simple relaxation can be strengthened until the polytope (or the convex body) it defines converges to the convex hull of the feasible 0-1 solutions. It is known that, for each of these hierarchies, if the original relaxation has $n$ variables, then the relaxation at level $t$ of the hierarchy can be solved optimally in time $n^{O(t)}$. Thus, to achieve improved approximations for a problem in polynomial (resp.\ sub-exponential time), we would like to know if we can beat the integrality gap of the natural relaxation by using a relaxation at level $O(1)$ (resp.\ $o(n/\log n)$) of some hierarchy.

Initially, this approach was refuted (for specific problems) by a long series of results showing that integrality gaps do not decrease at sufficiently low levels (see, e.g.~\cite{ABLT06,alekhnovich:arora:tourlakis,GMPT,CMM09}). Positive results have also recently emerged (e.g.~\cite{Chl07,MK10,BRS11,GS11}), where improved approximations were given using constant-level relaxations from various hierarchies. For a survey on both positive and negative results, see~\cite{CM-chapter}.

Recently, Karlin et al.~\cite{KMN11} considered how the use of {\iLP} and {\iSDP} hierarchies affects the integrality gap of relaxations for~\knapsack. This is of particular relevance to us, since their approach relies on a well-known PTAS for \knapsack\ which is similar in structure to our own combinatorial algorithm for~\setc. They showed that, while the \sa\ {\iLP} hierarchy requires $\Omega(n)$ levels to bring the integrality gap below $2-o(1)$, level $k$ of the \la\ {\iSDP} hierarchy brings the integrality gap down to $1+O(1/k)$. While we would like to emulate the success of their {\iSDP}-hierarchy-based approach (and give an alternative sub-exponential algorithm for \setc), 
 we note that for \setc\, Alekhnovich et al.~\cite{alekhnovich:arora:tourlakis} have shown that the {\iSDP} hierarchy $\LS_+$, due to Lov\'asz and Schrijver, requires $\Omega(n)$ levels to bring the integrality gap below $(1-o(1))\ln n$. Nevertheless, the comparison is not perfect, since the \la\ hierarchy is much stronger than $\LS_+$. In particular, previously it was not known whether $\LS_+$ also reduces the integrality gap for \knapsack\ (and indeed, the algorithm of Karlin et al.~\cite{KMN11} relied a powerful decomposition theorem for the \la\ hierarchy which does not seem to be applicable to $\LS_+$).

\subsection{Our Results}
To facilitate our lift-and-project based approach, we start by giving in Section~\ref{sec:combapprox sketch} a simple new sub-exponential time combinatorial algorithm for \setc\, which nearly matches the time-approximation tradeoff guarantee in~\cite{cygan:kowalik:wykurz}.
\begin{theorem}\label{thm:approx}
For any (not necessarily constant) $1 \leq d \leq n$, there is an $H_{n/d}$-approximation algorithm for \setc\ running in time ${\rm poly}(n,m) \cdot m^{O(d)}$.
\end{theorem}
While this theorem is slightly weaker than the previous best known guarantee, our algorithm is remarkably simple, and will be instrumental in designing a similar lift-and-project based \setc\ approximation. 
The algorithm is combinatorial and does not rely on linear programming techniques.
By choosing $d = n^\eps$, we get a sub-exponential time algorithm whose approximation guarantee is better than $\ln n$ by a constant factor.
Next in Section~\ref{sec:hierarchy}, we show that using level $d$ of the linear programming hierarchy of {\ls}~\cite{LS91}, we can match the performance of the algorithm we use to prove Theorem \ref{thm:approx},
though only if the ``lifting" is done after guessing the value of the objective function (using a binary search), and adding this as a constraint a priori.
In this case, the rounding algorithm is quite fast, and avoids the extensive combinatorial guessing of our first algorithm, while the running time is dominated by the time it takes to solve the {\iLP} relaxation.

On the other hand, without the trick of ``lifting the objective function", we show in Section~\ref{sec:sa_lower} that even the stronger {\iLP} hierarchy of {\sa}~\cite{SA90} has an integrality gap of at least $(1-\eps)\ln n$ at level $\Omega(n)$. Specifically, we show the following
\begin{theorem}\label{thm: SA linear tight integrality gap}
For every $0<\eps, \gamma \leq \frac{1}{2}$, and for sufficiently large values of $n$, there are instances of \setc\ on $n$ cover-sets (over a universe of $n$ items) for which the integrality gap of the level-$\lfloor \frac{\gamma(\eps-\eps^2)}{1+\gamma} n \rfloor$ \sa\ \iLP\ relaxation is at least $\frac{1-\eps}{1+\gamma}\ln n$.
\end{theorem}
As we have mentioned, the prospect of showing a positive result using {\iSDP} hierarchies is unlikely due to the work of Alekhnovich et al.~\cite{alekhnovich:arora:tourlakis} which gives a similar integrality gap for $\LS_+$.

For completeness, we also show in Section~\ref{sec:knapsack} that the {\la}-hierarchy-based \knapsack\ algorithm of Karlin et al.~\cite{KMN11} can be matched using the weaker $\LS_+$ hierarchy and a more complex rounding algorithm. Specifically, we show that the integrality gap of the natural relaxation for \knapsack\ can be reduced to $1+\eps$ using $O(\eps^{-3})$ rounds of $\LS_+$. This highlights a fundamental difference between \knapsack\ and \setc, despite the similarity between our combinatorial algorithm and the PTAS for \knapsack\ on which Karlin et al.\ rely. Formally, we prove the following:
\begin{theorem}\label{thm:knapsack_ig} The integrality gap of level $k$ of the $\LS_+$ relaxation for \knapsack\ is at most $1+O(k^{-1/3})$.
\end{theorem}

In what follows, and before we start the exposition of our results, we present in Section~\ref{sec: preliminaries on ls} the \ls\ system along with some well-known facts that we will need later on. We end with a discussion of future directions in Section~\ref{sec:conclusion}.


\section{Preliminaries on the \ls\ System}\label{sec: preliminaries on ls}

For any polytope $P$, this system begins by introducing a nonnegative %
 auxiliary variable $x_0$, so that in every constraint of $P$, constants are multiplied by $x_0$.  This yields %
 the cone 
 $K_0(P):=\{(x_0, x_0\bx)\mid x_0\geq 0~\&~\bx \in P\}$. 
 For an $n$-dimensional polytope $P$, the \ls\ system finds a hierarchy of nested cones $K_0(P) \supseteq K_1(P) \supseteq \ldots \supseteq K_n(P)$ (in the \iSDP\ variant, we will write $K^+_t(P)$), defined recursively, and which  enjoy remarkable algorithmic properties. In what follows, let $\powerset_k$ denote the space of vectors indexed by subsets of $[n]$ of size at most $k$, 
 and for any $\by\in\powerset_k$, define the moment matrix $Y^{[\by]}$ to be the square matrix with rows and columns indexed by sets of size at most $\lfloor k/2\rfloor$, where the entry at row $A$ and column $B$ is $y_{A\cup B}$. %
 Also we denote by $\be_0, \be_1, \ldots, \be_n$ the standard orthonormal basis of dimension $n+1$, such that $Y^{[\by]} \be_i$ is the $i$-th column of the moment matrix.

\begin{definition}[The \ls\ (\iLS) and \ls\ \iSDP\ (\iLSp) systems]\label{def: LS definition} 
Consider the conified polytope $K_0(P)$ defined earlier  (let us also write $K^+_0(P)=K_0(P)$). The level-$t$ \ls\ cone (relaxation or tightening) $K_t(P)$ (resp.\ $K^+_t(P)$) of $\LS$ (resp.\ $\LS_+$) is recursively defined as all $n+1$ dimensional vectors $(x_0, x_0\bx)$ for which there exist $\by \in \powerset_2$ such that $Y^{[\by]} \be_i, Y^{[\by]} \left( \be_0 - \be_i\right) \in K_{t-1}(P)$ (resp.\ $K^+_{t-1}(P)$) and $(x_0,x_0\bx) = Y^{[\by]} \be_0$. The level-$t$ \ls\ \iSDP\ tightening of $\LS_+$ asks in addition that $Y^{[\by]}$ is a positive-semidefinite matrix.
\end{definition}

In the original work of Lov\'asz and Schrijver~\cite{LS91} it is shown that the cone $K_n(P)$ (even in the \iLS\ system) projected on $x_0=1$ is exactly the integral hull of the original \iLP\ relaxation, while one can optimize over $K_t(P)$ in time $n^{O(t)}$, given that the original relaxation admits a (weak) polytime separation oracle. 
The algorithm in this section, as well as the one in Section~\ref{sec:knapsack}, both rely heavily on the following facts, which follow easily from the above definition:
\begin{fact}\label{fact: ls conditioning} For any vector $\bx\in[0,1]^n$ such that $(1,\bx)\in K_t(P)$, and corresponding moment vector $\by\in\powerset_2$, and for any $i\in[n]$ such that $x_i>0$, the rescaled column vector $\frac1{x_i}Y^{[\by]} \be_i$ is in $K_{t-1}(P)\cap\{(1,\bx')\mid\bx'\in[0,1]^n\}$.
\end{fact}
\begin{fact}\label{fact: ls integrality} For any vector $\bx\in[0,1]^n$ such that $(1,\bx)\in K_t(P)$, and  any coordinate $j$ such that $x_j$ in integral, for all $t'<t$, any vector $\bx'$ such that $(1,\bx')\in K_{t'}(P)$ derived from $\bx$ by one or more steps as in Fact~\ref{fact: ls conditioning}, we have $x_j'=x_j$.
\end{fact}

\section{Two Approaches for Proving Theorem~\ref{thm:approx}} \label{sec: two approaches for thm}
In the following sections, we let $(X, \mathcal S)$ denote a \setc\ instance with items $X$ and cover-sets $\mathcal S$ where each $S \in \mathcal S$ has cost $c(S)$.
We use $n$ to denote the number of items in $X$ and $m$ to denote the number of cover-sets in $\mathcal S$.

\subsection{Sketch of a Combinatorial Proof} \label{sec:combapprox sketch}

Recall that the standard greedy algorithm for approximating \setc\ iteratively selects the cover-set $S$ of minimum density $c(S)/|S \setminus \bigcup_{T \in \mathcal C} T|$ where
$\mathcal C$ is the collection of cover-sets already chosen. The approximation guarantee of this algorithm is $H_b$, where $b$ is the size of the largest cover-set. Our algorithm builds on this result simply by guessing up to $d$ cover-sets in the optimal solution before running the greedy algorithm. However, some of the cover-sets in $\mathcal S$ that were not guessed (and might still contain uncovered items) are discarded before running the greedy algorithm. Specifically, we discard the cover-sets that contain more than $\frac{n}{d}$ uncovered items after initially guessing the $d$ sets. We show that for some choice of $d$ sets, no remaining set in the optimum
solution covers more than $\frac{n}{d}$ uncovered items. Thus, running the greedy algorithm on the remaining sets  is actually an $H_{n/d}$-approximation.
The full description of the algorithm along with all details of the proof are in Appendix~\ref{sec:combapprox}

\subsection{Proof Based on the {\ls} System}
\label{sec:hierarchy}

In this section we provide an alternative {\iLP}-based approximation algorithm for \setc\ with the same performance as in Section~\ref{sec:combapprox sketch}, illustrating the power of the so-called lift-and-project systems. 
Consider the standard \iLP\ relaxation 
\begin{align}
{\rm minimize} \qquad& \sum_{S \in \mathcal S} c(s) x_S & \nonumber\\
{\rm subject~to} \qquad & \sum_{S \ni i} x_S  \geq 1 & \forall~ i \in X \label{eq:setcover-bound}\\
& 0 \leq x_S  \leq 1 & \forall~ S \in \mathcal S &\label{eq: setcover box}
\end{align}
for \setc. 
Now, consider the corresponding feasibility \iLP\ where instead of explicitly minimizing the objective function, we add the following bound on the objective function as a constraint (we will later guess the optimal value $q$ by binary search):%
\begin{equation}\label{equa: bound on cost setcover}
\sum_{S \in \mathcal S} c(s) x_S \leq q.
\end{equation}
We will work with the feasibility \iLP\ consisting only of constraints \eqref{eq:setcover-bound}, \eqref{eq: setcover box}
and \eqref{equa: bound on cost setcover} (and no objective function). 
Denote the corresponding polytope of feasible solutions by $P_q$.

In what follows we strengthen polytope $P_q$ using the \ls\ lift-and-project system. 
Next we 
show that the level-$d$ \ls\ relaxation $K_d(P_q)$ can give a $H_{\frac{n}{d}}$-factor approximation algorithm. We note here that applying the \ls\ system to the feasibility $P_q$ (which includes the objective function as a constraint) and not on the standard \iLP\ relaxation of \setc\ is crucial, since by Alekhnovich et al.~\cite{alekhnovich:arora:tourlakis} the latter \iLP\ has a very bad integrality gap even when strengthened by $\Omega(n)$ rounds of $\LS_+$ (which is even stronger than $\LS$).

To that end, let $q$ be the smallest value such that the level-$d$ \iLS\ tightening of $P_q$ is not empty (note that $q\leq \OPT$). The value $q$ can be found through binary search (note that in each stage of the binary search we attempt to check $K_d(P_{q'})$ for emptyness for some $q'$, which takes time $m^{O(d)}$). Our goal is to show that for this $q$ we can find a \setc\ of cost at most $q\cdot H_{\frac nd}$.

Let $\bx^{(d)}$ be such that $(1,\bx^{(d)})\in K_d(P_q)$. For any coordinate $i$ in the support of $\bx^{(d)}$ we can invoke Fact~\ref{fact: ls conditioning} and get a vector $\bx^{(d-1)}$ such that $(1,\bx^{(d-1)})\in K_{d-1}(P_q)$ and $\bx^{(d-1)}_i=1$. 
By Fact~\ref{fact: ls integrality}, by iterating this step, we eventually obtain a vector $\bx^{(0)}\in P_q$ which is integral in at least $d$ coordinates. Note that by constraint~\eqref{equa: bound on cost setcover}, this solution has cost at most $q$. %
 We refer to this subroutine as the \textit{Conditioning Phase}, which is realized in $d$ many inductive steps. 

If at some level $0 \leq i \leq d$, the sets whose coordinates in $\bx^{(i)}$ are set to 1 cover all universe elements $X$, we have solved the \setc\ instance with cost 
$$\sum_{S: \bx^{(i)}_S=1} C(S) \leq q \leq\OPT.$$
Otherwise, we need to solve a smaller instance of \setc\ defined by all elements $Y \subseteq X$ not already covered, using cover-sets $\mathcal T = \{ S\cap Y\mid \bx^{(0)}_S >0 \}$. We introduce some structure in the resulting instance $(Y, \mathcal T)$ of \setc\ by choosing the indices we condition on as in the proof of Lemma~\ref{lem:order}.  This gives the following Lemma, whose proof is similar to that of Lemma~\ref{lem:order}.
\begin{lemma}\label{lemma: small cover-sets in subinstance}
If at each step of the Conditioning Phase we choose the set $S$ in the support of the current solution $\bx^{(d')}$
containing the most uncovered elements in $X$, then for all $T \in \mathcal T$ we have $|T| \leq \frac{n}{d}$.
\end{lemma}
\begin{proof}

For $1 \leq i \leq d$ let $S_i$ denote the cover-set chosen in the Conditioning Phase for level $i$ and for $0 \leq i \leq d$ let $\mathcal C_i = \{S_d, S_{d-1}, \ldots, S_{i+1}\}$
(with $\mathcal C_d = \emptyset$).
We show that at every iteration $1 \leq i \leq d$ we have $|S \setminus \bigcup_{T \in \mathcal C_i} T| \leq \frac{n}{d-i+1}$ for every $S \in \mathcal S \setminus \mathcal C_i$
with $\bx^{(i)}_S > 0$.

For $1 \leq i \leq d$, let $\alpha_i := |S_i \setminus \bigcup_{T \in \mathcal C_i} T|$.
Since we chose the largest (with respect to the uncovered items) cover-set $S_i$ in the support of $\bx^{(i)}$ we have $|S' \setminus \bigcup_{T \in \mathcal C_i} T| \leq \alpha_i$
for every cover-set $S' \in \mathcal S \setminus \mathcal C_i$ with $\bx^{(i)}_{S'} > 0$. For $2 \leq i \leq d$ we also have
$\alpha_{i-1} \leq \alpha_i$ because $\mathcal C_i \subseteq \mathcal C_{i-1}$ and because we chose $S_i$ instead of $S_{i-1}$ in the Conditioning Phase for level $i$.

So, $\alpha_d \geq \alpha_{d-1} \geq \ldots \geq \alpha_1$. Now, each item $j$ covered by $\mathcal C_0$ contributes 1 to $\alpha_i$ for the earliest index $i$ for which $j \in S_i$,
so $\sum_{i=1}^d \alpha_i \leq n$.
This implies that $\alpha_i \leq \frac{n}{d-i+1}$ for each $1 \leq i \leq d$. Therefore, every set in the support of $\bx^{(0)}$ has at most
$\frac{n}{d}$ elements that are not already covered by $\mathcal C_0$. So, the instance $(Y, \mathcal T)$ has $|T| \leq \frac{n}{d}$ for any $T \in \mathcal T$.
\end{proof}

Let $\mathcal D$ be the collection of cover-sets chosen as in Lemma~\ref{lemma: small cover-sets in subinstance}. Observe that the vector $\bx^{(0)}$ projected on the cover-sets $\mathcal S \setminus \mathcal D$ that were \textit{not} chosen in the Conditioning Phase is feasible for the \iLP\ relaxation of the instance $(Y, \mathcal T)$. In particular, the cost of the \iLP\ is at most $q - \sum_{S \in \mathcal D} c(S)$, and by Lemma~\ref{lemma: small cover-sets in subinstance} all cover-sets have size at most $\frac{n}{d}$. By Theorem~\ref{thm:greedy}, the greedy algorithm will find a solution for $(Y, \mathcal T)$ of cost at most $H_\frac{n}{d} \cdot \left( q - \sum_{S \in \mathcal D} c(S) \right)$. Altogether, this gives a feasible solution for $(X, \mathcal S)$ of cost 
$$ H_\frac{n}{d} \cdot \bigg( q - \sum_{S \in \mathcal D} c(S) \bigg) 
 + \sum_{S \in \mathcal D} c(S) \leq H_\frac{n}{d} \cdot q \leq H_\frac{n}{d} \cdot \OPT.$$


\section{Linear \sa\ Integrality Gap for \setc} \label{sec:sa_lower}

The level-$\ell$ \sa\ relaxation is a tightened \iLP\ that can be derived systematically starting with any 0-1 \iLP\ relaxation. While in this work we are interested in tightening the \setc\ polytope, the process we describe below is applicable to any other relaxation. 

\begin{definition}[The \sa\ system]\label{def: SA definition}
Consider a polytope over the variables $y_1,\ldots, y_n$ defined by finitely many constraints (including the box-constraints $0\leq y_i \leq 1$). The level-$\ell$ \sa\ relaxation is an \iLP\ over the variables $\{y_A\}$ where $A$ is any subset of $\{1,2,\ldots,n\}$ of size at most $\ell+1$, and where $y_\emptyset =1$. For every constraint $\sum_{i=1}^n a_i y_i \geq b$ of the original polytope and for every disjoint $P,E \subseteq \{1,\ldots, n\}$ with $|P|+|E|\leq \ell$, the following is a constraint of the level-$\ell$ \sa\ relaxation
$$
\sum_{i=1}^n a_i   \sum_{\emptyset \subseteq T \subseteq E} (-1)^{|T|} y_{P \cup T \cup \{i\} }
 \geq b \sum_{\emptyset \subseteq T \subseteq E} (-1)^{|T|} y_{P \cup T }.
$$
\end{definition}

%
We will prove Theorem~\ref{thm: SA linear tight integrality gap} in this section. For this we will need 
 two ingredients: (a) appropriate instances, and (b) a solution of the \sa\ \iLP\ as described in Definition~\ref{def: SA definition}. Our hard instances are described in the following lemma, which is due to Alekhnovich et al.~\cite{alekhnovich:arora:tourlakis}.
\begin{lemma}[\setc\ instances with no small feasible solutions]~
\label{lem: hard SA instances}\\
For every $\eps> \eta>0$, and for all sufficiently large $n$, there exist \setc\ instances over a universe of $n$ elements and $n$ cover-sets, such that:\\
(i) Every element of the universe appears in exactly $(\eps-\eta)n$ cover-sets, and\\
(ii) There is no feasible solution that uses less than $\log_{1+\eps} n$ cover-sets.
\end{lemma}
In order to prove Theorem~\ref{thm: SA linear tight integrality gap} we will invoke Lemma~\ref{lem: hard SA instances} with appropriate parameters. Then we will define a vector solution for the level-$\ell$ \sa\ relaxation as described.

\begin{lemma}\label{lem: sa feasible solution}
Consider a \setc\ instance on $n$ cover-sets as described in Lemma~\ref{lem: hard SA instances}. Let $f$ denote the number of cover-sets covering every element of the universe. For $f\geq 3 \ell$, the vector $\by$ indexed by subsets of $\{1,\ldots,n\}$ of size at most $\ell+1$ defined as 
$ y_A := \frac{(f-\ell-1)!}{(f-\ell-1+|A|)!}, ~\forall A \subseteq \{1,\ldots,n\}, ~|A|\leq \ell+1$,
satisfies the level-$\ell$ \sa\ \iLP\ relaxation of the \setc\ polytope. 
\end{lemma}

The proof of Lemma~\ref{lem: sa feasible solution} involves a number of extensive calculations  which we give in Section~\ref{sec: proof of lemma feasible sa solution}. Assuming the lemma, 
 we are ready to prove Theorem~\ref{thm: SA linear tight integrality gap}. 

\begin{proof}[Proof of Theorem~\ref{thm: SA linear tight integrality gap}]
Fix $\eps>0$ and invoke Lemma~\ref{lem: hard SA instances} with $\eta = \eps^2$ to obtain a \setc\ instance on $n$ universe elements and $n$ cover-sets for which (i) every universe element is covered by exactly $(\eps-\eps^2)n$ cover-sets, and (ii) no feasible solution exists of cost less than $\log_{1+\eps} n$. Note that in particular (i) implies that in the \setc\ \iLP\ relaxation, every constraint has support exactly $f = (\eps-\eps^2)n$. 

Set $\ell = \frac{\gamma(\eps-\eps^2)}{1+\gamma} n$ and note that $f/\ell \geq 3$, since $\gamma \leq \frac{1}{2}$. This means we can define a feasible level-$\ell$ \sa\ solution as described in Lemma~\ref{lem: sa feasible solution}. The values of the singleton variables are set to 
$$ y_{\{i\}} = \frac{1}{(\eps-\eps^2)n-\ell} = \frac{1+\gamma}{(\eps-\eps^2)n}.$$
But then, the integrality gap is at least
$$ 	\frac{\OPT}{\sum_{i=1}^n y_{\{i\}}}
	\geq 
	\frac{\eps-\eps^2}{1+\gamma}\cdot{ \log_{1+\eps} n }
	=
	\frac{ \eps-\eps^2  }{(1+\gamma) \ln(1+\eps)} \ln n.
	$$
The lemma follows once we observe that $\ln(1+\eps) = \eps-\frac{1}{2}\eps^2 + \Theta(\eps^3)$.
\end{proof}

\subsection{Proof of Lemma~\ref{lem: sa feasible solution}}\label{sec: proof of lemma feasible sa solution}

Recall that $f$ denotes the support of every cover constraint in the \setc\ relaxation (i.e.\ the number of cover-sets every element belongs to in the instance described in Lemma~\ref{lem: hard SA instances}). Also recall that
$$ y_A := \frac{(f-\ell-1)!}{(f-\ell-1+|A|)!}, ~\forall A \subseteq \{1,\ldots,n\}, ~|A|\leq \ell+1.$$
To show feasibility of the above vectors, we need to study two types of constraints, i.e. the so called box-constraints
\begin{equation}\label{eqn: explicit box-constraints}
0 \leq  \sum_{\emptyset \subseteq T \subseteq E} (-1)^{|T|} y_{P \cup T } \leq 1, ~\forall P,E\subseteq \{1,\ldots,n\},~ |P|+|E|\leq \ell+1, 
\end{equation}
as well as the covering constraints
\begin{equation}\label{eqn: explicit cover-constraints}
\sum_{i \in D}   \sum_{\emptyset \subseteq T \subseteq E} (-1)^{|T|} y_{P \cup T \cup \{i\} }
 \geq \sum_{\emptyset \subseteq T \subseteq E} (-1)^{|T|} y_{P \cup T },~\forall P,E\subseteq \{1,\ldots,n\},~ |P|+|E|\leq \ell,
\end{equation}
where $D\subseteq \{1,\ldots,n\}$ is a set of $f$ many cover-sets covering some element of the universe. The symmetry of the proposed solutions allows us to significantly simplify the above expressions, by noting that for $|P|=p$ and $|E|=e$ we have 
$$\sum_{\emptyset \subseteq T \subseteq E} (-1)^{|T|} y_{P \cup T } 
	= 
	\sum_{t=0}^e (-1)^{t} \binom{e}{t} \frac{(f-\ell-1)!}{(f-\ell-1+p+t)!}
$$

For the sake of exposition, we show that our \iLP\ solution satisfies the two different kinds of constraints 
 in two different lemmata. Set $x = f- \ell-1+p$, and note that if $f\geq 3\ell$, then since $e\leq \ell$ we have $x \geq \frac{5}{3} e$. Thus, the feasibility of the box constraints~\eqref{eqn: explicit box-constraints} (for our solution) 
  is implied by the following combinatorial lemma. 

\begin{lemma}\label{lem: abstract box constraints}
For $x\geq \frac{5}{3} e$ we have 
$$ 0 \leq \sum_{t=0}^e (-1)^t \binom{e}{t} \frac{1}{(x+t)!} \leq \frac{1}{(x-p)!} $$
\end{lemma}

\begin{proof}
First we show the lower bound, for which we study two consecutive summands. We note that 
\begin{eqnarray*}
\binom{e}{2t} \frac{1}{(x+2t)!} - \binom{e}{2t+1} \frac{1}{(x+2t+1)!}
& = &	
\binom{e}{2t} \frac{1}{(x+2t)!} \left( 1 - \frac{\binom{e}{2t+1}}{\binom{e}{2t}} \frac{1}{x+2t+1}\right)  \\
& = &	
\binom{e}{2t} \frac{1}{(x+2t)!} \left( 1 - \frac{e-2t}{2t+1} \frac{1}{x+2t+1}\right)  \\
& \geq &	
\binom{e}{2t} \frac{1}{(x+2t)!} \left( 1 - \frac{e}{2t+1} \frac{1}{x+2t+1}\right)  \\
& \geq &	
\binom{e}{2t} \frac{1}{(x+2t)!} \left( 1 - \frac{e}{x} \right).
\end{eqnarray*}
Since $x\geq e$, every two consecutive summands add up to a non-negative value. Thus, if the number of summands is even (that is, if $e$ is odd), the lower bound follows, and if the number of summands is odd, the bound also follows, since for even $e$, the last summand is positive.


Now we show the upper bound. Note that, since $p$ does not appear in the sum, it suffices to bound the sum by the smaller value $\frac{1}{x!}$. This is facilitated by noting that 
$$ \sum_{t=0}^e (-1)^t \binom{e}{t} \frac{1}{(x+t)!} 
	=
  \frac{1}{x!} \sum_{t=0}^e (-1)^t \frac{1}{t!} \frac{\binom{x+e}{x+t}}{\binom{x+e}{e}}. $$
Hence, it suffices to show that 
\begin{equation}
\sum_{t=0}^e (-1)^t \frac{1}{t!} \binom{x+e}{x+t} \leq \binom{x+e}{e}. \label{equa: upper bound sum in box constraints}
\end{equation}
As before, we analyze the sum of two consecutive terms (this time in the above sum). 
 This is done in the next claim. 
\begin{claim}\label{cla: upper bound binom}
For $x \geq \frac{5}{3}e$ we have 
$$ \frac{1}{(2t)!} \binom{x+e}{x+2t} - \frac{1}{(2t+1)!} \binom{x+e}{x+2t+1}
\leq \binom{x+e}{x+2t} -  \binom{x+e}{x+2t+1}$$
\end{claim}
\begin{proof}
We divide both sides of the desired inequality by $\binom{x+e}{x+2t}$ to obtain the equivalent statement
\begin{eqnarray*}
&& \frac{1}{(2t)!} - \frac{1}{(2t+1)!} \frac{e-2t}{x+2t+1}
\leq 1 -  \frac{e-2t}{x+2t+1} \\
&\Leftrightarrow&  
	\frac{e-2t}{x+2t+1} \left(  1- \frac{1}{(2t+1)!} \right)
		\leq 
	1 - \frac{1}{(2t)!}
\end{eqnarray*}
Note that the above is tight for $t=0$. For $t>0$, and since $x \geq \frac{5}{3} e$, we have $\frac{e-2t}{x+2t+1} < \frac{3}{5}$ which is small enough to compensate for the worst-case ratio of the expressions involving factorials, which occurs for $t=1$. 
\end{proof}

Continuing our proof of Lemma~\ref{lem: abstract box constraints}, first suppose that $e$ is odd. Then Claim~\ref{cla: upper bound binom} implies that 
$$\sum_{t=0}^e (-1)^t \frac{1}{t!} \binom{x+e}{x+t} 
 \leq 
	\sum_{t=0}^e (-1)^t  \binom{x+e}{x+t} 
 =  \binom{x+e-1}{e} \leq \binom{x+e}{e}, $$
which gives the required condition~\eqref{equa: upper bound sum in box constraints} for odd $e$. If $e$ is a positive even integer, then again Claim~\ref{cla: upper bound binom} implies that 
$$\sum_{t=0}^e (-1)^t \frac{1}{t!} \binom{x+e}{x+t} 
 \leq 
	\sum_{t=0}^{e-1} (-1)^t  \binom{x+e}{x+t} + \frac{1}{e!}
 =  \binom{x+e-2}{e} + \frac{1}{e!} < \binom{x+e}{e}. $$
Finally, if $e=0$ then~\eqref{equa: upper bound sum in box constraints} holds with equality. This concludes 
 the proof.
\end{proof}

Now we turn our attention to cover constraints \eqref{eqn: explicit cover-constraints}. 
 Recall that the values $y_A$ only depend on the size of $A$. Since $f$ and $\ell$ are fixed in the context of Lemma~\ref{lem: sa feasible solution}, for $|P|=p$ and $|E|=e$ 
 we can define
$$ H_{e,p}:=\sum_{\emptyset \subseteq T \subseteq E} (-1)^{|T|} y_{P \cup T }  =   \sum_{t=0}^e (-1)^{t} \binom{e}{t} \frac{(f-\ell-1)!}{(f-\ell-1+p+t)!}.$$
Thus 
 the left-hand-side of \eqref{eqn: explicit cover-constraints} (for fixed $P$ and $E$) involves only 
 expressions of the form $H_{e,p}, H_{e,p+1}$, or 0,
depending on the relationship between the sets $P,E$ and $\{i\}$, while the right-hand-side equals $H_{e,p}$. 

More concretely, let $|D \cap P| = p_1$, $|D \cap E| = e_1$, $|P\setminus D| = p_0$ and $|E\setminus D| = e_0$, where $p_0+p_1 = |P|=p$ and $e_0+e_1 = |E|=e$, and recall that $E$ and $P$ are disjoint. Then observe that the $p_1$ 
 indices in $D\cap P$ each contribute $H_{e,p}$ to the left-hand-side of \eqref{eqn: explicit cover-constraints}. In addition, the $e_1$ indices $i\in D\cap E$ each contribute $0$ to the left-hand-side. This follows because each $T \subseteq E \setminus \{i\}$ can be paired with $T' := T \cup \{i\} \subseteq E$ and the terms $y_{P \cup T \cup \{i\}}$ and
$y_{P \cup T' \cup \{i\}}$ are identical, while they appear in the sum with opposite signs. 
Finally, the remaining $f-p_1-e_1$ indices contribute each $H_{e,p+1}$. Overall, for our proposed solution, Constraint~\eqref{eqn: explicit cover-constraints} can be rewritten as
$$
(f - p_1 - e_1 ) H_{e,p+1} + p_1 H_{e,p} \geq H_{e,p}
$$
Clearly, for $p_1>0$, the above constraint is satisfied. Hence, we may assume that $|P \cap D| = \emptyset$, and so $|P|=p_1=p$. Note also that the value $e_1$ does not affect $H_{e,p}$, thus the above inequality holds for all $e_1\leq e$ iff it holds for $e_1=e$ and $e_0=0$. 
To summarize, to show that Constraint~\eqref{eqn: explicit cover-constraints} 
 is satisfied, we need only show that 
\begin{equation}\label{eqn: explicit cover-constraints evaluated}
(f - e ) H_{e,p+1}  \geq H_{e,p}
\end{equation}
for $H_{e,p} = \sum_{t=0}^e (-1)^{t} \binom{e}{t} \frac{(f-\ell-1)!}{(f-\ell-1+p+t)!}$, and $e+p \leq \ell$. Note that the value $(f-\ell-1)!$ in the numerator appears in both sides. Also $f \geq 3 \ell$ and $e \leq \ell$ implies that $e<2\ell\leq f- \ell$. Finally recall that in~\eqref{eqn: explicit cover-constraints} we have 
$p+e =|P|+|L|\leq \ell$.  Hence, to show that~\eqref{eqn: explicit cover-constraints evaluated} (and thus Constraint~\eqref{eqn: explicit cover-constraints}) is satisfied, 
it remains only to show the following lemma. 

\begin{lemma}\label{lem: abstract cover constraints}
For $e<f-\ell$ and $p+e\leq \ell$  we have
$$ (f - e ) \sum_{t=0}^e (-1)^{t} \binom{e}{t} \frac{1}{(f-\ell+p+t)!}  - \sum_{t=0}^e (-1)^{t} \binom{e}{t} \frac{1}{(f-\ell-1+p+t)!} \geq 0$$
\end{lemma}

\begin{proof}
We rewrite the left-hand-side as
\begin{align*}
\sum_{t=0}^e (-1)^{t} &\binom{e}{t} \frac{1}{(f-\ell-1+p+t)!} 
\left( \frac{f-e}{f-\ell+p+t} - 1 \right ) \\
&=  \sum_{t=0}^e (-1)^{t} \binom{e}{t} \frac{1}{(f-\ell-1+p+t)!} 
 \frac{\ell - (p+e) - t}{f-\ell+p+t}\\
&=
\left( \ell - (p+e) \right) \sum_{t=0}^e (-1)^{t} \binom{e}{t} \frac{1}{(f-\ell+p+t)!} 
- \sum_{t=0}^e (-1)^{t} \binom{e}{t} \frac{t}{(f-\ell+p+t)!} \\
\end{align*}
The first sum of the last expression is non-negative, as it is a multiple of the box-constraints we have already proven. So is its coefficient, since $p+e \leq \ell$. Note that for $e=0$ and $p=\ell$ the above expression equals $0$ (and in particular Constraint~\eqref{eqn: explicit cover-constraints} is 
tight). 
 Next we show that for other values of $p,e$ (either $p<\ell$ or $e>0$) 
 the above expression is strictly positive. For this it suffices to show that $\sum_{t=0}^e (-1)^{t} \binom{e}{t} \frac{t}{(f-\ell+p+t)!} \leq 0$.

We proceed again by analyzing every two consecutive terms. We observe that for $t\geq 1$ we have
\begin{align*}
- \binom{e}{2t-1} &\frac{2t-1}{(f-\ell+p+2t-1)!} +  \binom{e}{2t} \frac{2t}{(f-\ell+p+2t)!} \\
&=
\frac{\binom{e}{2t-1}}{(f-\ell+p+2t-1)!}
	\left( -(2t-1) + \frac{e-2t+1}{f-\ell+p+2t} \right) \\
&<
\frac{\binom{e}{2t-1}}{(f-\ell+p+2t-1)!}
	\left( -(2t-1) + \frac{e}{f-\ell} \right) \\
&< \frac{\binom{e}{2t-1}}{(f-\ell+p+2t-1)!}
	\left( -1 + \frac{e}{f-\ell} \right) 
\end{align*}
which is non-positive, since $f-\ell \geq e$. This argument shows that every two consecutive summands of $\sum_{t=0}^e (-1)^{t} \binom{e}{t} \frac{t}{(f-\ell+p+t)!}$ add up to a non-positive value. If $e$ is even, then we are done, while for odd $e$ the unmatched summand (for $t=e$) is negative, and so the lemma follows.
\end{proof}


\section{An $\LS_+$-based PTAS for \knapsack} \label{sec:knapsack}

We consider the \knapsack\ problem: We are given $n$ items which we identify with the integers~$[n]$, and each item $i\in[n]$ has some associated (nonnegative) reward $r_i$  
 and cost (or size) $c_i$. 
  The goal is to choose a set of items which fit in the knapsack, i.e.\ whose total cost does not exceed some bound $C$, so as to maximize the total reward. 
In what follows we will use the \iLP\ $\{ \max \sum_{i=1}^n r_i x_i: ~\sum_{i=1}^n c_i x_i  \leq C ~\&~0 \leq x_i  \leq 1 \forall~ i \in [n] \}$, which is the natural relaxation for \knapsack. 

Denote the polytope associated with the \knapsack\ \iLP\ relaxation 
by $P$. We will consider the \iSDP\ derived by applying sufficiently many levels of $\LS_+$ (as defined in Section~\ref{sec:hierarchy}) to the above \iLP. That is, for some $\ell>0$, we consider the \iSDP\ 
\begin{align*}
{\rm maximize} \qquad& \sum_{i=1}^n r_i x_i & \\
{\rm subject~to} \qquad & (1,\bx)\in K^+_\ell(P).
\end{align*}

There is a well-known simple greedy algorithm for \knapsack: Sort the items by decreasing order of $r_i/c_i$, and add them to the knapsack one at a time until the current item does not fit. The following lemma (which is folklore) relates the performance of the greedy algorithm to the value of the \iLP\ relaxation $P$:
\begin{lemma}\label{lem:greedy} Let $\bx$ be a solution to $P$, and $R_G$ be the reward given by the greedy algorithm. Then $\sum_{i}r_ix_i\leq R_G+\max_{i}r_i.$
\end{lemma}
This gives a trivial bound of 2 on the integrality gap, assuming that $c_i\leq C$ for all $i$ (that is, that each item can be a solution on its own), since we then have $R_G+\max_{i}r_i\leq 2\opt$.\footnote{Here, as before, $\opt$ denotes the optimal 0-1 solution.} We will see later that the above assumption can essentially be enforced, that is, that we can ignore items with reward greater than $C$ (see Lemma~\ref{lem:knapsack-0-1}). Lemma~\ref{lem:greedy} has the following easy corollary: Consider the above greedy algorithm, with the modification that we first add all items which have $x_i=1$ and discard all items which have $x_i=0$. Then the following holds:
\begin{corollary}\label{cor:greedy} Let $\bx$ be a solution to $P$, and let $R'_G$ be the total reward given by the above modified greedy algorithm. Then $\sum_{i}r_ix_i\leq R'_G+\max_{i:{0<x_i<1}}r_i.$
\end{corollary}

We will show that, for any constant $\eps>0$, there is a constant $L_\eps$ such that the \iSDP\ relaxation for \knapsack\ arising from level $L_\eps$ of $\LS_+$ has integrality gap at most $1+O(\eps)$. For the Lasserre hierarchy, this has been shown for level $1/\eps$~\cite{KMN11}. 
 We will show this for $L_\eps=1/\eps^3$ in the case of $\LS_+$.

Our rounding algorithm will take as input the values of the \knapsack\ instance $(r_i)_i$, $(c_i)_i$, and $C$, an optimal solution $\bx$ s.t.\ $(1,\bx)\in K^+_\ell(P)$ (for some level $\ell>0$, initially $\ell=L_{\eps}$), and parameters $\eps$ and $\rho$. 
  The parameter $\rho$ is intended to be the threshold $\eps\cdot\opt$ in the set $S_{\eps\opt}=\{i\mid r_i > \eps\cdot\opt\}.$ 
  Rather than guessing a value for $\opt$, though, we will simply try all values of $\rho\in\{r_i\mid i\in[n]\}\cup\{0\}$ and note that for exactly one of those values, the set $S_{\eps\opt}$ coincides with  the set %
 $\{i\mid r_i > \rho\}$ (also note that $\rho$ is a parameter of the rounding, and not involved at all in the \iSDP\ relaxation).
 
The intuition behind our rounding algorithm is as follows: As we did for \setc, we would like to repeatedly ``condition" on setting some variable to $1$, by using Fact~\ref{fact: ls conditioning}. If we condition only on (variables corresponding to) items in $S_{\eps\opt}$, then after at most $1/\eps$ iterations, the \iSDP\ solution will be integral on that set, and then by Corollary~\ref{cor:greedy} the modified greedy algorithm will give a $1+O(\eps\opt)$ approximation relative to the value of the objective function (since items outside $S_{\eps\opt}$ have reward at most $\eps\opt$). The problem with this approach (and the reason why \iLP\ hierarchies do not work), is the same problem as for \setc: the conditioning step does not preserve the value of the objective function. While the optimum value of the \iSDP\ is at least $\opt$ by definition, after conditioning, the value of the new solution may be much smaller than $\opt$, which then makes the use of Corollary~\ref{cor:greedy} meaningless. The key observation is that the use of \iSDP{s} ensures that we can choose some item to condition on without any decrease in the objective function (see Lemma~\ref{lem:increase-obj-fun}).\footnote{Note that this is crucial for a maximization problem like \knapsack, while for a minimization problem like \setc\ it does not seem helpful (and indeed, by the integrality gap of Alekhnovich et al.~\cite{alekhnovich:arora:tourlakis}, we know it does not help).} A more refined analysis shows that we will be able condition {\em specifically on items in $S_{\eps\opt}$} without losing too much. Counter-intuitively, we then need to show that the algorithm does not perform an unbounded number of conditioning steps which {\em increase} the objective value (see Lemma~\ref{lem:recurse-depth}). Our  rounding algorithm {\bf{KS-Round}} is described in Algorithm~\ref{alg:knapsack}, while the performance guarantee is described in Section~\ref{sec:knapsack_app}. 
\begin{algorithm*}[ht]
  \caption{~\bf{KS-Round}$((r_i)_i,(c_i)_i,C,\bx,\eps, \rho
)$} \label{alg:knapsack} 
\begin{algorithmic}[1] 
\State Let $\by\in\powerset_2$ be the moment vector associated with $(1,\bx)$.
\State Let $S_\rho \leftarrow \{i\mid r_i>\rho\}$, and let $S^b \leftarrow \{i\mid x_i =b\}$ for $b=0,1$. \label{step:knapsack:S-rho}
\If{$S_\rho\subseteq S^0\cup S^1$}
\State Run the modified greedy algorithm. \label{step:knapsack:0-1}
\ElsIf{$\displaystyle\sum_{i\in S_\rho\setminus S^1}r_ix_i < \eps\cdot\sum_{i=1}^n r_ix_i$}
\State Run the modified greedy algorithm on items in $([n]\setminus S_\rho)\cup S^1$.  \label{step:knapsack:small-rewards}
\ElsIf{there is some $i\in S_\rho\setminus(S^0\cup S^1)$ s.t.\ $\displaystyle\sum_{j=1}^n r_jy_{\{i,j\}}\geq (1-\eps^2)x_i\cdot\sum_{j=1}^nr_jx_j$}
\State Run {\bf{KS-Round}}$((r_i)_i,(c_i)_i,C,\frac1{x_i}Y^{[\by]} \be_i,
\eps, \rho)$. \Comment See Fact~\ref{fact: ls conditioning}
\label{step:knapsack:recurse-S-rho}
\Else\State Choose $i\in[n]\setminus (S_\rho\cup S^0)$ s.t.\ $\displaystyle\sum_{j=1}^n r_jy_{\{i,j\}}>(1+\eps^3)x_i\cdot\sum_{j=1}^nr_jx_j$ \Comment See Lemma~\ref{lem:strict-increase} \label{step:knapsack:choose-increase}
\State Run {\bf{KS-Round}}$((r_i)_i,(c_i)_i,C,\frac1{x_i}Y^{[\by]} \be_i,
\eps, \rho)$. \Comment See Fact~\ref{fact: ls conditioning}
\label{step:knapsack:recurse-increase}
\EndIf
\end{algorithmic}
\end{algorithm*}


\subsection{Analysis of Algorithm~\ref{alg:knapsack}}
\label{sec:knapsack_app}

Before we analyze the performance guarantee of the algorithm, there is one more simple fact about both $\LS$ and $\LS_+$ which we will use in this section.

\begin{fact}\label{fact:LS-for-x_i=1}
Given a solution $(1,\bx)\in K_{\ell}(P)$ for some $t\geq 1$ and corresponding moment vector $\by\in\powerset_2$, if $x_i=1$ for some $i\in[n]$, then $Y^{[y]}\be_i=Y^{[y]}\be_0(=(1,\bx))$.
\end{fact} 
This fact follows easily from the fact that $Y^{[\by]} \left( \be_0 - \be_i\right) \in K_{t-1}(P)$, and since if $x_i=y_{\{i\}}=1$ then the first entry in the above vector is $0$, which for the conified polytope $K_{t-1}(P)$ only holds for the all-zero vector.

Now, consider the set $S_\rho$ defined in Step~\ref{step:knapsack:S-rho}. As we have pointed out, for the appropriate choice of $\rho$, this set coincides with the set $S_{\eps\opt}$. Note that, by Corollary~\ref{cor:greedy}, if all the $x_i$ values in $S_\rho$ are integral, then Step~\ref{step:knapsack:0-1} returns a solution with value $R'_G$ satisfying
$$\opt \geq R'_G \geq \sum_ir_ix_i - \max_{i\not\in S_\rho}r_i \geq \sum_ir_ix_i - \eps\cdot\opt.$$
Now, if $\bx$ is the original \iSDP\ solution given to the rounding algorithm, this gives an upper bound of $1+\eps$ on the integrality gap (as well as a $(1+\eps)$-approximation). However, in Steps~\ref{step:knapsack:recurse-S-rho} and~\ref{step:knapsack:recurse-increase}, we recurse with a new \iSDP\ solution (to a lower level in the $\LS_+$ hierarchy). Thus, our goal will be to arrive at an \iSDP\ solution which is integral on $S_\rho$ (or assigns so little weight to $S_\rho$ that we can ignore it), but we need to show that the objective function does not decrease too much during these recursive steps. To show this, we will rely crucially on the following easy lemma, which is also the only place where we use positive-semidefiniteness. 

\begin{lemma}\label{lem:increase-obj-fun} Let $(1,\bx)$ be a solution to $K^+_\ell(P)$ for some $\ell\geq 1$, with the corresponding moment vector $\by$. Then the solution satisfies $$\sum_{i=1}^nr_i\sum_{j=1}^nr_jy_{\{i,j\}}\geq\left(\sum_{i=1}^n r_ix_i\right)^2 \left(= \sum_{i=1}^n r_i \sum_{j=1}^n x_ir_jx_j\right).$$
\end{lemma}
\begin{proof} By the positive semidefiniteness of the moment matrix $Y^{[\by]}$, we have $\ba^{\top} Y^{[\by]} \ba\geq 0$ for any vector $\ba\in\Real^{n+1}$. Then the lemma follows immediately from this inequality, by letting $\ba=(a_i)_i$ be the vector defined by $a_i=r_i$ for $i\in[n]$ and $a_0=-\sum_{i=1}^nr_ix_i$.
\end{proof}

Thus, there is some item $i\in[n]$ on which we can condition (by taking the new solution $\frac1{x_i}Y^{[\by]} \be_i$, 
 without any decrease in the value of the objective function. Moreover, using the above lemma, we can now show that the algorithm is well-defined (assuming we have start with an \iSDP\ at a sufficiently high level of $\LS_+$ for all the recursive steps).

\begin{lemma}\label{lem:strict-increase}
If Step~\ref{step:knapsack:choose-increase} is reached, then there exists an item $i\in[n]\setminus (S_\rho\cup S^0)$ satisfying 
\begin{equation}\label{eq:lem:strict-increase}\displaystyle\sum_{j=1}^n r_jy_{\{i,j\}}>(1+\eps^3)x_i\cdot\sum_{j=1}^nr_jx_j.
\end{equation}
\end{lemma}
\begin{proof}
Note that if Step~\ref{step:knapsack:choose-increase} is reached then we must have 
\begin{equation}\label{eq:lem-increase-1}\displaystyle\sum_{i\in S_\rho\setminus S^1}r_ix_i \geq \eps\cdot\sum_{i=1}^n r_ix_i.
\end{equation}
Moreover, for all $i\in S_\rho\setminus(S^0\cup S^1)$ we have $\displaystyle\sum_{j=1}^n r_jy_{\{i,j\}}< (1-\eps^2)x_i\cdot\sum_{j=1}^nr_jx_j$. In particular, we have 
\begin{equation}\label{eq:lem-increase-2}\sum_{i\in S_\rho\setminus S^1}r_i\sum_{j=1}^n r_jy_{\{i,j\}}< (1-\eps^2)\sum_{i\in S_\rho\setminus S^1}r_ix_i\cdot\sum_{j=1}^nr_jx_j.
\end{equation} 
Thus (noting that all $i\in S^0$ contribute nothing to the following sums), we have
\begin{align*} \sum_{i\in ([n]\setminus (S_\rho\cup S^0))\cup S^1}r_i\sum_{j=1}^n r_jy_{\{i,j\}} &= \sum_{i=1}^n r_i\sum_{j=1}^n r_jy_{\{i,j\}} - \sum_{i\in S_\rho\setminus S^1}r_i\sum_{j=1}^n r_jy_{\{i,j\}}\\
&\geq \left(\sum_{i=1}^n r_ix_i\right)^2 - \sum_{i\in S_\rho\setminus S^1}r_i\sum_{j=1}^n r_jy_{\{i,j\}} &\text{by Lemma~\ref{lem:increase-obj-fun}}\\
&> \left(\sum_{i=1}^n r_ix_i\right)^2 - (1-\eps^2)\sum_{i\in S_\rho\setminus S^1}r_ix_i\cdot\sum_{j=1}^nr_jx_j &\text{by~\eqref{eq:lem-increase-2}}\\
&=\sum_{i\in ([n]\setminus S_\rho)\cup S^1}r_ix_i\cdot\sum_{j=1}^nr_jx_j + \eps^2\sum_{i\in S_\rho\setminus S^1}r_ix_i\cdot\sum_{j=1}^nr_jx_j\\
&\geq\sum_{i\in ([n]\setminus S_\rho)\cup S^1}r_ix_i\cdot\sum_{j=1}^nr_jx_j + \eps^3\sum_{i=1}^nr_ix_i\cdot\sum_{j=1}^nr_jx_j &\text{by~\eqref{eq:lem-increase-1}}\\
&\geq (1+\eps^3)\cdot \sum_{i\in ([n]\setminus S_\rho)\cup S^1}r_ix_i\cdot\sum_{j=1}^nr_jx_j.
\end{align*}

Therefore, there is some $i\in ([n]\setminus (S_\rho\cup S^0))\cup S^1$ satisfying~\eqref{eq:lem:strict-increase}. We only need to show that $i\not\in S^1$, that is, that $x_i < 1$. However, by Fact~\ref{fact:LS-for-x_i=1}, 
 if $x_i=1$ then $y_{\{i,j\}}=x_j$, making inequality~\eqref{eq:lem:strict-increase} impossible for such $i$.
\end{proof}

Now that we have shown the algorithm to be well-defined, let us start to bound the depth of the recursion. We will use the following lemma to show that after a bounded number of recursive calls in Step~\ref{step:knapsack:recurse-S-rho}, the \iSDP\ solution becomes integral on $S_\rho$:
\begin{lemma}\label{lem:knapsack-0-1} Let $(1,\bx)$ be a solution to $K^+_\ell$ for some $\ell\geq 1$. Then for all items $i\in[n]\setminus S^1$ such that $c_i > C - \sum_{j\in S^1}c_j$, we have $x_i=0$.
\end{lemma}
\begin{proof} Suppose, for the sake of contradiction, that there is some item $i$ satisfying this property, with $0<x_i<1$. By Fact~\ref{fact:LS-for-x_i=1}, 
for every $j\in S^1$ we have $y_{\{i,j\}} = x_i$. However, by Fact~\ref{fact: ls conditioning}, 
 the vector $\frac1{x_i}Y^{[\by]} \be_i$ 
  satisfies the constraints of $K^+_{\ell-1}$, and in particular the capacity constraint. 
But this is a contradiction, since we have $${\textstyle\frac1{x_i}}\sum_{j=1}^nc_jy_{\{i,j\}}=c_i+\sum_{j\in S^1}c_j+\sum_{j\in[n]\setminus(S^1\cup\{i\})}c_jx_j > C.$$
\end{proof}

Since no $1/\eps$ items from $S_{\eps\opt}$ can fit simultaneously in the knapsack, we have the following:

\begin{corollary}\label{cor:knapsack-0-1} For $\rho$ s.t.\ $S_\rho=S_{\eps\opt}$, the recursion in Step~\ref{step:knapsack:recurse-S-rho} cannot be repeated $\lceil 1/\eps \rceil$ times. Moreover, if this step is repeated $(\lceil 1/\eps\rceil-1)$ times, then after these recursive calls, we have $x_i\in\{0,1\}$ for all $i\in S_\rho$.
\end{corollary}

We can now bound the total depth of the recursion in the algorithm.

\begin{lemma}\label{lem:recurse-depth} For all positive $\eps\leq0.3$, the algorithm performs the recursion in Step~\ref{step:knapsack:recurse-increase} less than $\lceil1/\eps^3\rceil$ times.
\end{lemma}
\begin{proof}
This follows by tracking the changes to the value of the objective function. Each time Step~\ref{step:knapsack:recurse-S-rho} is performed, the value of the objective function changes by a factor at least $(1-\eps^2)$, while each time Step~\ref{step:knapsack:recurse-increase} is performed, this value changes by a factor greater than $(1+\eps^3)$. Assuming, for the sake of contradiction, that Step~\ref{step:knapsack:recurse-increase} is performed $\lceil1/\eps^3\rceil$ times, then by Corollary~\ref{cor:knapsack-0-1} the total change is at least a factor $(1-\eps^2)^{\lceil 1/\eps\rceil-1}(1+\eps^3)^{\lceil1/\eps^3\rceil}>2$. However, this is a contradiction, since initially the value of the objective function is at least $\opt$, and the integrality gap is always at most $2$.
\end{proof}

Finally, we can prove that the algorithm gives a $(1+O(\eps))$-approximation, thus bounding the integrality gap. Theorem~\ref{thm:knapsack_ig} follows immediately from
the following.

\begin{theorem}\label{thm:knapsack-alg} For sufficiently small $\eps>0$, given an optimal (maximum feasible) solution $(1,\bx)$ to $K^+_\ell(P)$ for $\ell=O(1/\eps^3)$, there is a value $\rho\in\{r_i\mid i\in[n]\}\cup\{0\}$ such that algorithm {\bf{KS-Round}} finds a solution to the \knapsack\ instance with total value (reward) at least $(1-O(\eps))\cdot\sum_{i=1}^nr_ix_i$.
\end{theorem}

\begin{proof}[Proof of Theorem~\ref{thm:knapsack-alg}.] It is easy to see that for (exactly one) $\rho\in\{r_i\mid i\in[n]\}\cup\{0\}$ we have $S_\rho=S_{\eps\opt}$. Choose this value of $\rho$.

Note that the only time the value of the objective function can decrease is in Step~\ref{step:knapsack:recurse-S-rho}. Let $\Phi$ be the initial value of the objective function. Then by Corollary~\ref{cor:knapsack-0-1}, at the end of the recursion, the new value of the objective function will be at least $(1-\eps^2)^{1/\eps}\cdot\Phi = (1-O(\eps))\Phi$. Thus, it suffices to show the theorem relative to the final value of the objective function (as opposed to $\Phi$)\footnote{In fact, this will also show that the integrality gap remains at most $1+O(\eps)$ throughout the algorithm (as opposed to $2$, which we used in the proof of Lemma~\ref{lem:recurse-depth}), thus bounding the depth of the recursion by $O(1/\eps^2)$ rather than $O(1/\eps^3)$.}.

Let us consider the algorithm step-by-step. If the algorithm terminates in Step~\ref{step:knapsack:0-1}, then since for all $i\in[n]\setminus S_{\eps\opt}$ we have $r_i\leq\eps\cdot\opt$, by Corollary~\ref{cor:greedy} the rounding loses at most a $(1+O(\eps))$ factor relative to the value of the objective function.

If the algorithm terminates in Step~\ref{step:knapsack:small-rewards}, then similarly, the rounding loses at most a $(1+O(\eps))$-factor relative to
$\sum_{i\in[n]\setminus S_\rho}r_ix_i>(1-\eps)\sum_{i=1}^nr_ix_i$.

Finally, note that by Corollary~\ref{cor:knapsack-0-1} and Lemma~\ref{lem:recurse-depth}, we have sufficiently many levels of the hierarchy to justify the recursive Steps~\ref{step:knapsack:recurse-S-rho}  and~\ref{step:knapsack:recurse-increase}, and note that the choice of item $i$ in Step~\ref{step:knapsack:choose-increase} is well-defined by Lemma~\ref{lem:increase-obj-fun}.
\end{proof}

\section{Conclusion} \label{sec:conclusion}

The known sub-exponential time approximation algorithms place \setc\ in a distinct category from 
other optimization problems like \threesat\ in the following sense. Though one can achieve provably hard approximation factors for \setc\ in sub-exponential time, Moshkovitz and Raz~\cite{moshkovitz:raz} show that improving over the easy $\frac87$-approximation for \threesat{}~\cite{johnson} by any constant requires time $2^{n^{1-o(1)}}$, assuming the ETH.

Rather, \setc\ lies in the same category of problems as \chrom\ and \clique\ both of which admit $n^{1-\eps}$ approximations in time $2^{\tilde{O}(n^\eps)}$ (by partitioning the graph into sets of size $n^{\eps}$ and solving the problem optimally on these sets), despite the known $n^{1-o(1)}$ hardness of approximation for both problems~\cite{hastad,FK,zuckerman}. This may also be taken as evidence that the recent subexponential time algorithm of Arora, Barak, and Steurer~\cite{arora:barak:steurer} for \ug\ does not necessarily imply that \ug\ is not hard, or not as hard as other NP-hard optimization problems.

Finally, turning to lift-and-project methods, we note here that our $\LS_+$-based algorithm for \knapsack\ shows that in some instances reduced integrality gaps which rely heavily on properties of the \la\ hierarchy can be achieved using the weaker $\LS_+$ hierarchy. This raises the question of whether the problems discussed in the recent series of \la-based approximation algorithms~\cite{BRS11,GS11,RT} also admit similar results using $\LS_+$. On the flip side, it would also be interesting to see whether any such problems have strong integrality gap lower bounds for $\LS_+$, which would show a separation between the two hierarchies.


\paragraph{Acknowledgements} We would like to thank Dana Moshkovitz for pointing out the blowup in her \setc\ reduction. We would also like to thank Mohammad R. Salavatipour for preliminary discussions on sub-exponential time approximation algorithms in general, and  Claire Mathieu for insightful past discussions of the \knapsack-related results in~\cite{KMN11}. Finally, we would like to thank Marek Cygan for bringing~\cite{cygan:kowalik:wykurz} to our attention. 


\bibliographystyle{plain}
\bibliography{setcover}

\appendix

\section{A Combinatorial Approximation Algorithm} \label{sec:combapprox}

The following result is well-known and follows from standard dual-fitting techniques.
\begin{theorem}[Chv\'atal~\cite{chvatal}, paraphrased]\label{thm:greedy}
The greedy algorithm is a polynomial-time $H_b$-approximation where $b$ is the size of the largest cover-set in $\mathcal S$.
Moreover, it finds a solution whose cost is at most $H_b$ times the optimum value of the standard {\iLP} relaxation.
\end{theorem}

Motivated by the former Theorem, we present next our combinatorial algorithm (sketched in Section~\ref{sec:combapprox sketch}) in full details. 

\begin{algorithm*}[ht]
  \caption{~An Improved Approximation for \setc} \label{alg:sub-exp}
\begin{algorithmic}[1] 
\State $\mathcal C \leftarrow \mathcal S$
\For {each collection $\mathcal D$ of at most $d$ sets in $\mathcal S$}
\State Let $\mathcal T \leftarrow \{S \in \mathcal S : |S \setminus \bigcup_{T \in \mathcal D} T| \leq \frac{n}{d}\}$ 
\State {\bf if} $\mathcal T \cup \mathcal D$ cannot cover $X$ {\bf then} skip to the next iteration
\While {$\mathcal D$ does not cover $X$} \label{step:greedystart}
\State Let $S \leftarrow \arg \min_{S \in \mathcal T} \frac{c(S)}{|S \setminus \bigcup_{T \in \mathcal D} T|}$
\State $\mathcal D \leftarrow \mathcal D \cup \{S\}$
\EndWhile \label{step:greedyend}
\State {\bf if} $\mathcal D$ is cheaper than $\mathcal C$ {\bf then} $\mathcal C \leftarrow \mathcal D$
\EndFor
\State{\bf return} $\mathcal C$
\end{algorithmic}
\end{algorithm*}
The outer loop of Algorithm~\ref{alg:sub-exp} is reminiscent of the steps in many PTASes that ``guess'' the largest items to be used  (for example, many \knapsack\ variants such as in~\cite{chekuri:khanna}).
This is almost a correct interpretation of the guesswork done in Algorithm~\ref{alg:sub-exp}. The following simple structural result will help articulate this notion in the proof
of the main result. 

\begin{lemma}\label{lem:order}
If $\mathcal C$ is solution to a \setc\ instance $(X, \mathcal S)$, then there is an ordering $S_1, \ldots, S_k$ of the cover-sets in $\mathcal C$
such that $|S_i \setminus \bigcup_{j=1}^{i'-1} S_j| \leq \frac{n}{i'}$ for any $1 \leq i' \leq i \leq k$.
\end{lemma}
\begin{proof}
Order $\mathcal C$ greedily: Iteratively select the set in $\mathcal C$ that covers the most uncovered items. Then, by definition of this ordering, for the first $i'$ sets we have $n\geq |\bigcup_{j=1}^{i'}S_j|\geq i'\cdot|S_{i'} \setminus \bigcup_{j=1}^{i'-1} S_j|$. Moreover, for $i>i'$, since $S_{i'}$ is chosen instead of $S_i$, we must have $|S_{i} \setminus \bigcup_{j=1}^{i'-1} S_j|\leq|S_{i'} \setminus \bigcup_{j=1}^{i'-1} S_j|\leq\frac{n}{i'}$.
\end{proof}
We are now ready to give all details for the combinatorial proof of Theorem~\ref{thm:approx}.
\begin{proof}[Proof of Theorem~\ref{thm:approx}]
Let $\mathcal C^*$ be an optimal solution of cost $\OPT$ and size, say, $k$. Furthermore, let $S^*_1, \ldots, S^*_k$ be an ordering of the cover-sets in $\mathcal C^*$ with the property guaranteed by
Lemma~\ref{lem:order}. If $k \leq d$ then the set $\mathcal C^*$ will be considered in some iteration so Algorithm~\ref{alg:sub-exp} will actually find an optimal solution.

Otherwise, consider the iteration of the outer loop that guessed $\mathcal D = \{S^*_1, \ldots, S^*_d\}$. Consider the \setc\ instance $(Y, \mathcal T)$ with items
$Y := X \setminus \bigcup_{T \in \mathcal D} T$ and cover-sets
$\mathcal T := \{ S \setminus \bigcup_{T \in \mathcal D} T : S \in \mathcal S, |S \setminus \bigcup_{T \in \mathcal D} T| \leq \frac{n}{d}\}$. The cost of the
cover-sets $S \in \mathcal T$ in this instance should be equal
to their original costs (before we subtracted $\bigcup_{T \in \mathcal D} T$).
Each cover-set in $\mathcal T$ has size at most $\frac{n}{d}$ so, by Theorem
\ref{thm:greedy}, the greedy algorithm applied to this instance is an $H_{n/d}$-approximation. Lemma~\ref{lem:order} guarantees
that the restriction of the cover-sets $S^*_{d+1}, \ldots, S^*_k$ to $X \setminus \bigcup_{T \in \mathcal D} T$ have size at most $\frac{n}{d}$ so these restrictions form a valid solution
for $(Y, \mathcal T)$ of cost $\sum_{i=d+1}^k c(S^*_i)$. Thus, the optimum solution to the \setc\ instance $(Y, \mathcal T)$ has cost $\sum_{i=d+1}^k c(S^*_i)$.
Steps~\ref{step:greedystart}-\ref{step:greedyend} act just like the greedy algorithm on this instance $(Y, \mathcal T)$ which, by Theorem~\ref{thm:greedy},
adds a total of $H_{\frac{n}{d}} \cdot \sum_{i=d+1}^k c(S^*_i)$ to the cost of $\mathcal D$. The overall cost of the final solution $\mathcal D$ is at most
$\sum_{i=1}^d c(S^*_i) + H_{\frac{n}{d}} \cdot \sum_{i=d+1}^k c(S^*_i) \leq H_{\frac{n}{d}} \cdot \sum_{i=1}^k c(S^*_i) = H_{\frac{n}{d}} \cdot \OPT.$
Finally, note that there are $m^{O(d)}$ iterations of the outer loop, each taking polynomial time.
\end{proof}

\end{document}